\documentclass[11pt, reqno, a4paper, final]{amsart}

\usepackage[english]{babel}
\usepackage[T1]{fontenc}
\usepackage{amsmath,amsfonts,amssymb,amsthm,amscd}
\usepackage{eucal}
\usepackage{array}
\usepackage{color}
\usepackage{bbm}
\usepackage{graphicx}
\usepackage{hyperref}
\usepackage{a4wide}										 	
\usepackage{mathrsfs}									  
\usepackage[all]{xy}										
\usepackage{url}												
\usepackage{verbatim,epsfig}
\usepackage[notref,notcite,draft]{showkeys}   
\usepackage{xfrac} 

\newcounter{claim-counter}

\theoremstyle{plain}

\newtheorem{thm}{Theorem}[section] 
 
\newtheorem{cor}[thm]{Corollary}

\newtheorem{lem}[thm]{Lemma}
\newtheorem{lemma}[thm]{Lemma}
\newtheorem{prop-defi}[thm]{Definition \& Proposition}
\newtheorem{prop}[thm]{Proposition}

\newtheorem*{thm*}{Theorem}
\newtheorem*{prop*}{Proposition}
\newtheorem*{cor*}{Corollary}

\usepackage{thmtools}
\declaretheorem[style=theorem,name={Theorem}]{theoremletter}

\theoremstyle{definition}
\newtheorem{defi}[thm]{Definition}

\newtheorem{rem}[thm]{Remark}

\newtheorem*{claim*}{Claim}

\newcommand{\NN}{{\mathbb N}}
\newcommand{\ZZ}{{\mathbb Z}}
\newcommand{\BB}{{\mathbb B}}

\newcommand{\RR}{{\mathbb R}}
\newcommand{\CC}{{\mathbb C}}

\newcommand{\MM}{{\mathbb M}}

\newcommand{\Z}{{\mathcal Z}}
\newcommand{\R}{{\mathcal R}}

\renewcommand{\H}{\mathscr{H}}
\renewcommand{\L}{{\mathscr L}}

\newcommand{\G}{{\mathscr G}}

\newcommand{\htens}{\bar{\otimes}}

\newcommand{\id}{\operatorname{id}}

\renewcommand{\d}{\operatorname{d}\hspace{-0.03cm}}

\newcommand{\T}{{\mathscr{T}}}

\renewcommand{\leq}{\leqslant}

\renewcommand{\i}{{\operatorname{i}}}

\renewcommand{\S}{{\mathcal{S}}}
\renewcommand{\restriction}{\rvert}
\newcommand{\Sp}{{\operatorname{Sp}}}
\newcommand{\Linfty}{\operatorname{L^\infty}}

\newcommand{\rS}{\mathrm{S}}

\newcommand{\Cdot}{\raisebox{-0.80ex}{\scalebox{2.2}{$\cdot$}}}

\allowdisplaybreaks

\title[Measure equivalence for non-unimodular groups]{Measure equivalence for non-unimodular groups}

\author{Juhani Koivisto, David Kyed and Sven Raum}
\address{Juhani Koivisto, Department of Mathematics and Computer Science, University of Southern Denmark, Campusvej 55, DK-5230 Odense M, Denmark}
\email{juhanialex.koivisto@gmail.com }

\address{David Kyed, Department of Mathematics and Computer Science, University of Southern Denmark, Campusvej 55, DK-5230 Odense M, Denmark}
\email{dkyed@imada.sdu.dk}

\address{Sven Raum,
Department of Mathematics, Stockholm University, SE-106 91 Stockholm, Sweden}
\email{raum@math.su.se}


\begin{document}

\begin{abstract}
We undertake a comprehensive study of measure equivalence between general locally compact, second countable groups, providing operator algebraic and ergodic theoretic reformulations, and complete the classification of amenable groups within this class up to measure equivalence.

\end{abstract}

\maketitle

\section{Introduction}
The notion of measure equivalence for discrete groups was  invented by Gromov \cite{gromov-asymptotic-invariants} as a measure theoretic analogue of quasi isometry, and in the works of Furman  \cite{furman-gromovs-measure-equivalence, Furman-OE-rigidity}, it was used  to prove strong rigidity results for lattices in higher rank simple Lie groups. In recent years, there has been an increasing focus on ergodic theoretic aspects of non-discrete groups \cite{bowen-hoff-ioana, BFS-integrable, carderi-le-maitre, KPV, bader-rosendal, BDV17} and the present paper provides a contribution to this  programme by analyzing the role of measure equivalence in the context of locally compact, second countable groups.
The case of unimodular groups was  treated in \cite{KKR17}, but for non-unimodular groups it has been an open problem to localize a good definition of measure equivalence. Although Deprez and Li \cite{deprez-li-ME} proposed a definition of measure equivalence between non-unimodular groups, it remained open whether or not this notion is compatible with stable orbit equivalence.  Further, measure equivalence was shown to preserve the Haagerup property and weak amenability in \cite[Theorem 6.1]{deprez-li-ME}, but its behaviour with respect to amenable groups was not clarified, although the concise classification of amenable groups up to measure equivalence  is a distinct feature of this notion for unimodular groups. In the present paper, we first show that the definition of measure equivalence suggested in \cite{deprez-li-ME} has the properties that one would expect a good notion of measure equivalence to have, in that it allows for the passage to strict, free and ergodic couplings (see Proposition \ref{prop:strict-free-couplings}). Secondly, we show that it connects well with other established equivalence relations among locally compact  groups, by means of the following theorem whose discrete version appeared in \cite[Theorem 3.3]{Furman-OE-rigidity}.

\begin{theoremletter}[see {Theorem \ref{thm:me-and-soe-equivalence}}]
For non-discrete, locally compact, second countable groups $G$ and $H$ the following are equivalent.
\begin{itemize}
\item[(i)] $G$ and $H$ are measure equivalent.
\item[(ii)] $G $ and $H $ admit  orbit equivalent, essentially free, ergodic,  probability measure preserving actions on standard Borel probability spaces.
\item[(iii)] $G$ and $H$ admit essentially free, ergodic probability measure preserving actions on standard Borel probability spaces for which the cross section equivalence relations associated with some choice of cross sections are stably orbit equivalent.
\end{itemize}\end{theoremletter}

A theorem of Furman \cite{furman-gromovs-measure-equivalence},  building on the works of Ornstein-Weiss \cite{ornstein-weiss} and Connes-Feldman-Weiss \cite{Connes-Feldman-Weiss}, shows that all countably infinite amenable groups are measure equivalent and in \cite{KKR17} the authors extended this to the class of unimodular, amenable, locally compact, second countable groups. Utilizing the deep classification results for injective factors \cite{connes-type-III-classification, con76, uffe-bicentralizer}, we now complete the classification of amenable,  locally compact, second countable groups up to measure equivalence by proving  the following theorem.

\begin{theoremletter}[{see Theorem \ref{thm:ME-classification-of-amenable}}]
The class of amenable, locally compact, second countable groups is stable under measure equivalence and consists of the following three measure equivalence classes: compact groups, non-compact unimodular amenable groups and non-unimodular amenable groups.
\end{theoremletter}
We remark that the fact that  amenability is preserved under measure equivalence is a direct consequence of our Theorem A together with \cite[Proposition 5.10]{zimmer-82}

\subsection*{Acknowledgments}
DK and JK gratefully acknowledge the financial support from  the Villum Foundation (grant no.~7423) and from the Independent Research Fund Denmark (grant no.~7014-00145B).
The authors would also like to thank Stefaan Vaes for sharing his insights at various stages of the project and Maxime Gheysens and Nicolas Monod  for bringing the reference \cite{gheysens-monod} to their attention.

\section{Preliminaries}
\subsection{Group actions and orbit equivalence}
In this section we recall the notion of orbit equivalence and stable orbit equivalence of group actions. Our setting will always be that of locally compact, second countable (lcsc) groups, and if $G$ is such a group, $\lambda_G$ will denote a left Haar measure on $G$, and $\Delta_G\colon G\to \RR_+:=]0,\infty[$ its modular function.
We will assume familiarity with the basics on standard Borel (probability) spaces, group actions on these, essential freeness, ergodicity etc., all of which is treated in detail in \cite{KKR17} and the references therein.

\begin{defi}\label{def:oe-and-soe}
\label{def:soe}
Let $G$ and $H$ be lcsc groups  acting non-singularly on standard Borel probability spaces $(X,\mu)$ and $(Y,\nu)$ respectively.
\begin{itemize}
\item The two actions are said to be \emph{orbit equivalent} (OE) if there exist a non-singular Borel map $\Delta\colon X\to Y$ and  conull Borel subsets $X_0\subset X$ and $Y_0\subset Y$ such that  $\Delta$ restricts to a Borel isomorphism $\Delta_0\colon X_0 \to Y_0$ 
with the property that
\begin{align*}
\Delta_0(G.x \cap X_0)=H.\Delta_0(x)\cap Y_0 \text{ for all } x\in X_0.
\end{align*}
\item The two actions are said to be \emph{stably orbit equivalent} (SOE) if there exist  Borel subsets $A\subset X$ and $B\subset Y$ of strictly positive measure such that the saturations $G.A\subset X$ and $H.B\subset Y$ are conull, and there exists a Borel isomorphism $\Delta \colon A\to B$  with  $\Delta_*(\mu\restriction_{A} )\sim \nu\restriction_{B}$ such  that
\begin{align*}
\Delta(G.a \cap A)=H.\Delta(a)\cap B \text{ for all } a\in A.
\end{align*}

\end{itemize}
\end{defi}
\begin{rem}
We leave out the routine verifications that OE and SOE do indeed form equivalence relations. Note that both OE and SOE are preserved under passing to a conull invariant subset; i.e.~if $G\curvearrowright (X,\mu)$ and $H\curvearrowright (Y,\nu)$ are non-singular Borel actions and $X'\subset X, Y'\subset Y$ are conull Borel subsets invariant under $G$ and $H$, respectively, then the two original actions are OE (respectively SOE) if and only if the restricted actions to $X'$ and $Y'$ are OE (respectively SOE).  \\
\end{rem}

Another central  notion to the paper is that of a countable, non-singular Borel equivalence relation on a standard Borel probability space $(X,\mu)$. Recall that  this is an equivalence relation $\R\subset X\times X$ which is Borel in the product $\sigma$-algebra, has countable equivalence classes, and is non-singular in the following sense: for every pair of Borel subsets $A, B\subset X$ and every Borel isomorphism $\psi\colon A\to B$ with graph in $\R$ we have $\psi_*\mu\restriction_A\sim \mu\restriction_B$. The relation $\R$ is probability measure preserving (pmp) if $\psi_*(\mu(A)^{-1}\mu\restriction_A)=\mu(B)^{-1}\mu\restriction_B$ for all such $\psi$. Lastly, the equivalence relation $\R$ is said to be ergodic if for any Borel subset $A\subset X$ its saturation $\R[A]:=\{x\in X\mid \exists a\in A: (x,a)\in \R\}$ is either null or conull. 
Recall that, by \cite[Theorem 1]{FM1}, for any countable equivalence relation $\R$ there exists a countable group $\Gamma$ and a non-singular action $\Gamma\curvearrowright (X,\mu)$ for which $\R$ is the orbit equivalence relation $\R_{\Gamma\curvearrowright X}$, and measure preservingness and ergodicity of the relation $\R$ then agree with the corresponding classical notions for the action $\Gamma\curvearrowright X$. We may therefore also consider orbit equivalence and stable orbit equivalence of countable equivalence relations by simply defining them as in Definition \ref{def:oe-and-soe} via their implementations as orbit equivalence relations of group actions.  Note also that by Furman's results \cite[Theorem D]{Furman-OE-rigidity}, it is not always possible to find an essentially free group action realizing a given equivalence relation. Countable Borel equivalence relations are examples of so-called measure groupoids, which will be discussed in the following section. \\

\subsection{Measure groupoids}
In this section we give a brief introduction to the theory of measure groupoids, which provides a convenient setting in which to treat countable Borel equivalence relations and actions of lcsc groups simultaneously.  Since this is primarily a tool for our further investigations we will not elaborate on the basics of measure groupoids, but refer the reader to the detailed treatments in \cite{hahn-haar-measure-groupoids} and \cite{ramsey}. 
Let us first recall the notion of a measure groupoid from \cite[Definition 2.3]{hahn-haar-measure-groupoids}.  We refer to this reference for details.  A \emph{measure groupoid} is a pair $(\G, C)$, where $\G$ is a Borel groupoid whose Borel structure is analytic and $C$ is an invariant symmetric measure class.  We recall that if $(\G, C)$ is a measure groupoid with range and source maps $r(\gamma):=\gamma\gamma^{-1}$ and $s(\gamma):=\gamma^{-1}\gamma$ then, by  \cite{hahn-haar-measure-groupoids},  $\G$ admits a  Haar measure $(\lambda, \mu)$; that is, $\lambda\in C$ is a measure on $\G$ and $\mu\in r_*(C)$ is a measure on the unit space $U_{\G}:=r(\G)$, and $(\lambda,\mu)$ satisfies a natural invariance condition \cite[Definition 3.11]{hahn-haar-measure-groupoids}. As  for groups, such a Haar measure gives rise to a modular function $\Delta\colon \G\to \RR_+^*$, which is a strict homomorphism on an \emph{inessential reduction}; i.e.~the measure groupoid defined as $\G\restriction_{X_0}:=\{\gamma\in \G\mid r(\gamma), s(\gamma)\in X_0 \}$ for a conull Borel subset $X_0\subset U_\G$. The groupoid $\G$ is said to be \emph{unimodular} if it admits a Haar measure for which the modular function $\Delta$ is similar to the constant function 1; i.e.~there exists a Borel function $f\colon U_\G\to ]0,\infty[$ such that
$
\Delta=(f\circ r)^{-1}(f\circ s) \text{ almost everywhere}.
$
If this is the case for some choice of Haar measure then it is the case for any choice of Haar measure as shown in \cite[Corollary 3.14]{hahn-haar-measure-groupoids}, and unimodularity is therefore an intrinsic property of the measure groupoid $(\G, C)$. Lastly, we recall that two measure groupoids, $(\G, C)$ and $(\H,D)$, are said to be \emph{isomorphic} if there exist inessential reductions $\G_0$ and $\H_0$ of $\G$ and $\H$, respectively, and a Borel isomorphism $\Phi\colon \G_0 \to \H_0$ which is an algebraic isomorphism of groupoids \cite[Section 3]{feldman-hahn-moore} such that $\Phi_*C=D$. In this case we will write $\G\simeq \H$ and refer to the isomorphism $\Phi\colon \G_0 \to \H_0$ as a \emph{strict} isomorphism from $\G$ to $\H$. \\

 The example of importance for our purposes is the \emph{action groupoid} stemming from a non-singular action $G\curvearrowright (X,\mu)$ of an lcsc group $G$ on a standard Borel probability space $(X,\mu)$. In that situation, setting $\G:=G\times X$ with multiplication
\[
(g_2, x_2)\cdot (g_1,x_1):=(g_2g_1,x_1) \text{ if } x_2=g_1.x_1,
\]
turns $(\G, [\lambda_G\times \mu])$ into a measure groupoid, denoted  by $G\ltimes X$ in what follows. The range and source maps are given by $s(g,x)=(e_G,x)$ and $r(g,x)=(e_G,g.x)$ and, as is customary, we shall often notationally identify $U_\G$ with $X$.  The following two lemmas are probably well known to experts in the field, but we were unable to find  a suitable reference and have therefore included their proofs for the benefit of the reader.
\begin{lemma}\label{lem:unimod-of-groupoid}
If $G\curvearrowright (X,\mu)$ is a pmp action of an lcsc group $G$ on a standard Borel probability space then the action groupoid $G\ltimes X$ is unimodular if and only if $G$ is unimodular.
\end{lemma}
\begin{proof}
The measure $(\lambda_G\times \mu, \mu)$ is a Haar measure for $G\ltimes X$ with modular function $\Delta(g,x)=\Delta_G(g)$ \cite[Example 1.11]{hahn-reg-rep} so if $G$ is unimodular then clearly so is $G\ltimes X$. Conversely, if $G\ltimes X$ is unimodular and $\Delta$ denotes the modular function associated with $(\lambda_G\times \mu,\mu)$ then there exists  a Borel function $f\colon X\to ]0,\infty[$ such that
\[
\Delta_G(g)=\Delta (g,x)=f(g.x)^{-1}f(x) \text{ for almost all } (g,x)\in G\times X.
\]
The Borel sets $X_n:=\{x\in X\mid f(x)\leq n\}$  exhaust $X$, so from a certain point on these must have positive measure, and we now fix one such $X_n$. Assume, towards a contradiction, that $G$ is not unimodular. Then there exists an open, precompact subset $U\subset G$ such that $\Delta_G(g)<1$ for all $g\in U$, and for almost all $(g,x)\in U\times X_n$  we therefore have
\begin{align}\label{eq:alpha-preserves}
f(g^{-1}.x)=\Delta_G(g^{-1})^{-1}f(x)< f(x)\leq n.
\end{align}
The map $\alpha\colon G\times X \to G\times X $ given by $\alpha(g,x)=(g,g^{-1}.x)$ is measure preserving and maps a conull subset of $U\times X_n$ into $U\times X_n$ by \eqref{eq:alpha-preserves}. We therefore get
\begin{align*}
  \int_{G\times X} 1_{U\times X_n} (g,x) f(x) \d\lambda_G(g) \d\mu(x)
  & =\int_{G\times X} 1_{U\times X_n} (g,g^{-1}.x) f(g^{-1}.x) \d\lambda_G(g) \d\mu(x) \\
  &= \int_{G\times X} 1_{U\times X_n} (g,x) \Delta_G(g)^{-1} f(x) \d\lambda_G(g) \d\mu(x)\\
  & >  \int_{U\times X_n}  f(x) \d\lambda_G(g) \d\mu(x),
\end{align*}
and since both integrals are finite, 
we have reached the desired contradiction.
\end{proof}

We will also need  the notion of \emph{similarity} between two measure groupoids $\G$ and $\H$. 
Following  \cite[Section 4]{feldman-hahn-moore},   we denote by $\T$ the unique (up to isomorphism), transitive measure groupoid  with uncountable orbits. One realization of $\T$ is obtained by taking the action groupoid arising from the action of the unit circle $\rS^1$ on itself by rotations. Similarly, $\Z$ denotes the unique transitive measure groupoid  with countable orbits, which can be realized by the transitive equivalence relation on $\ZZ$ considered as a measure groupoid. Two measure groupoids $(\G,C)$ and $(\H,D)$ are said to be \emph{similar} if $\G\times \T$ and $\H\times \T$ are isomorphic. 
 Note that this definition agrees with the one given in \cite{feldman-hahn-moore},  by \cite[Theorem 6.1]{ramsey-topologies-on-measured-groupoids} and \cite[Theorem 4.6]{feldman-hahn-moore}. 
 Also note that for action groupoids arising from essentially free actions of non-discrete lcsc groups similarity actually is the same as isomorphism \cite[Corollary 5.8]{feldman-hahn-moore}.
Lastly, we remark that the notion of \emph{orbitally concrete} measure groupoids featured in \cite{feldman-hahn-moore} was later shown to be automatic in \cite[Theorem 5.6]{ramsey-topologies-on-measured-groupoids}.\\

The following lemma connects the notions of OE and SOE with isomorphism and similarity of measure groupoids.

\begin{lemma}\label{lem:OE-and-iso}
Let $G$ and $H$ be lcsc groups acting essentially freely, ergodically and non-singularly  on standard Borel probability spaces $(X,\mu)$ and $(Y,\nu)$, respectively, and denote by $\G$ and $\H$ the corresponding action groupoids. Then the following hold:
\begin{itemize}
\item[(i)] The actions are orbit equivalent if and only if $\G$ and $\H$ are isomorphic.
\item[(ii)] If the actions are stably orbit equivalent then $\G$ and $\H$ are similar.
\item[(iii)] If $G$ and $H$ are either both infinite discrete or both non-discrete and  $\G$ and $\H$ are similar then their actions are stably orbit equivalent.
\end{itemize}
\end{lemma}

\begin{proof}
Since OE and SOE are preserved by passing to invariant, conull, Borel subsets and groupoid isomorphism and similarity are preserved by passing to inessential reductions, we may  assume that both actions are  free.  We first prove (i). Assume that the actions are OE and choose conull Borel subsets $X_0\subset X$ and $Y_0\subset Y$ and a Borel isomorphism  $\Delta_0\colon X_0 \to Y_0$ such that 
\begin{align}\label{oe-condition}
\Delta_0(G.x \cap X_0)=H.\Delta_0(x)\cap Y_0 \ \text{for all $x\in X_0$}.
\end{align}
 We now show that the inessential reduction $\G_0:=\G\restriction_{X_0}$ and $\H_0:=\H\restriction_{Y_0}$ are strictly isomorphic, and hence that $\G\simeq \H$ as desired. Let $(g,x)\in \G_0$ be given. Thus, $x, g.x\in X_0$ and by \eqref{oe-condition} and  freeness  there exists a unique $c(g,x)\in H$ such that $\Delta_0(g.x)= c(g,x)\Delta_0(x)$, and this defines a Borel map $c\colon \G_0 \to H$. 
 Again by freeness of the actions, the map $c$ is a cocycle in the sense that for $((g_2,x_2),(g_1,x_1))\in \G_0^{(2)}$ one has $g_1x_1=x_2$ and $c(g_2g_1,x_1)=c(g_2,g_1.x_1)c(g_1,x_1)$.  Define $\Phi\colon \G_0 \to \H_0$  by $\Phi(g,x):=(c(g,x),\Delta_0(x))$. 
 The defining relation $\Delta_0(g.x)= c(g,x)\Delta_0(x)$ shows that $\Phi\times \Phi$ maps $\G_0^{(2)}$ into $\H_0^{(2)}$ and
 the cocycle property for $c$  implies that 
 $\Phi$ is a strict homomorphism of groupoids from $\G_0$ to $\H_0$ which is easily seen to be bijective.
Lastly, we need to see that $\Phi$ pushes the class of the Haar measure on $\G_0$ onto the class of that of $\H_0$. However, since $\Phi$ is an algebraic isomorphism of groupoids, $\Phi_*(\lambda_G \times \mu)$ is a  Haar measure on $\H_0$, so to see that $\Phi_*(\lambda_G \times \mu)\sim \lambda_H\times \nu$ it suffices \cite[Proposition 3.4]{feldman-hahn-moore} to show that $(s_{\H})_*\Phi_*[\lambda_G\times \mu]=[\nu]$, which follows since  $s_\H\circ \Phi(g,x)=\Delta_0(x)$ and $\Delta_{0*}(\mu)\sim \nu$ by assumption.
Thus, $\Phi$ is an isomorphism of measure groupoids. \\

Conversely, assume that $\G\simeq \H$ and let $\Phi\colon \G_0 \to \H_0$ be a strict isomorphism  between inessential reductions of $\G$ and $\H$  with $\Phi_*[\lambda_G\times \mu]=[\lambda_H\times \nu]$. Denote by $X_0\subset X$ and $Y_0\subset Y$ the unit spaces of $\G_0$ and $\H_0$. Since $\Phi$ is an isomorphism, it maps units to units so $\Phi(e_G,x)=(e_H,\Delta_0(x))$ for a unique $\Delta_0(x)\in Y_0$ and we now prove that the Borel isomorphism $\Delta_0\colon X_0 \to Y_0$ defined like this is an orbit equivalence.  
 For $(g,x)\in \G_0$, denote $\Phi(g,x)$ by $(h,y)$ and note that
\[
(e_H, \Delta_0(x))=\Phi(e_G,x)=\Phi((g,x)^{-1}\cdot (g,x))=\Phi(g,x)^{-1}\cdot\Phi(g,x)=(h^{-1},h.y)\cdot(h,y)=(e_H, y),
\]
So defining $c:=\pi_H \circ \Phi$ we see that $\Phi(g,x)
=(c(g,x),\Delta_0(x))$, and  a direct computation, using that $\Phi$ is a homomorphism, reveals that $c$ satisfies the cocycle property. 
We now have
\begin{align*}
(e_H, \Delta_0(g.x))
=\Phi((g,x)\cdot (g,x)^{-1})
=(c(g,x), \Delta_0(x))\cdot (c(g,x),\Delta_0(x))^{-1}
=(e_H,c(g,x)\Delta_0(x))
\end{align*}
Thus,  $\Delta_0(g.x)=c(g,x).\Delta_0(x)$ and since $\Phi(g,x)=(c(g,x),\Delta_0(x))\in \H_{0}$ we have that $c(g,x).\Delta_0(x)\in Y_0$; i.e.~$\Delta_0$ satisfies \eqref{oe-condition}.
The last thing to be shown is that  $\Delta_{0*}\mu\sim \nu$, which can be seen using that $\Phi_*(\lambda_G\times \mu)\sim \lambda_H\times \nu$.\\

The proof of (ii) follows  the same line of reasoning and we therefore only sketch the argument: If the two actions are stably orbit equivalent, we obtain  Borel subsets $A\subset X$ and $B\subset Y$ of strictly positive measure and with  conull saturations,  as well as a Borel isomorphism $\Delta\colon A\to B$ pushing 
$[\mu\restriction_A]$ to $[\nu\restriction_B ]$ and satisfying the analogue of \eqref{oe-condition} on these subsets. As above, we get a cocycle $c\colon \G\restriction_A \to H$ and a map
$\Phi\colon \G\restriction_A \to \H\restriction_B$ defined by $\Phi(g,a)=(c(g,a), \Delta(a))$ which is seen to be a strict isomorphism of groupoids. Since $\G$ and $\H$ are ergodic and $A,B$ have positive measure, \cite[Proposition 4.9]{feldman-hahn-moore} shows that $\G\restriction_A$ and $\H\restriction_B$ are reductions of $\G$ and $\H$, respectively, and since the two reductions are isomorphic we conclude that $\G$ and $\H$ are similar. 
To prove (iii), assume first that both $G$ and $H$ are discrete and that $\G$ and $\H$ are similar. By \cite[Proposition 4.14 and Theorem 4.12]{feldman-hahn-moore}, this means that $\G\times \Z\simeq \H \times \Z$. Realizing $\Z$ as the equivalence relation of a discrete group action and applying (i), this shows that $\R_{G\curvearrowright X}\times \Z$ and $\R_{H\curvearrowright Y}\times \Z$ are orbit equivalent and by \cite[Theorem 3]{FM1}  we conclude that the original actions are stably orbit equivalent. If both $G$ and $H$ are non-discrete and $\G$ and $\H$ are similar, then by 
\cite[Theorem 5.6]{feldman-hahn-moore} there exist measure groupoids $\G_0$ and $\H_0$ with discrete orbits so that $\G\simeq \G_0\times \T$ and $\H\simeq \H_0\times \T$. In particular, $\G_0$ and $\H_0$ are reductions (cf.~\cite[Definition 3.5]{feldman-hahn-moore}) of $\G$ and $\H$, respectively, and since $\G$ and  $\H$ are similar, so are $\G_0$ and $\H_0$. By \cite[Proposition 4.14 and Theorem 4.12]{feldman-hahn-moore} this implies that $\G_0\times \Z\simeq \H_0\times \Z$ and since $\T\simeq \Z\times \T$, we have
\[
\G\simeq \G_0\times \Z\simeq  \G_0\times \Z \times \T \simeq \H_0\times \Z \times \T \simeq \H_0\times \T\simeq \H,
\]
so in this case $\G$ and $\H$ are not only similar, but actually isomorphic, and the associated actions therefore  orbit equivalent by (i).
\end{proof}

\begin{rem}\label{rem:non-free-iso-implies-OE}
Note that, in the proof of Lemma \ref{lem:OE-and-iso}, freeness of the actions was not used when proving that isomorphism of the action groupoids implies orbit equivalence of the actions. 
\end{rem}

\subsection{Cross section equivalence relations}\label{subsec:cross-sections}

In this section we recall the notion of a cross section for an action of an lcsc group, which is originally due to Forrest \cite{Forrest} (treated later in \cite{feldman-hahn-moore}; see also \cite{KPV} and \cite{carderi-le-maitre} for more recent treatments), and extend \cite[Proposition 4.3]{KPV} to the setting of non-unimodular groups. 
Let $G$ be an lcsc  group  and let $(X,\mu)$ be a standard Borel probability space endowed with a non-singular, essentially free action $\theta\colon G \curvearrowright (X,\mu)$. 
By \cite[Proposition 2.10]{Forrest} (see also \cite[Theorem 4.2]{KPV}) one may find a Borel subset $X_0\subset X$ and an open neighbourhood of the identity $U\subset G$ such that: 
\begin{itemize}
\item[(i)] the restricted action map $\theta\restriction \colon U\times X_0 \to X$ is injective, and
\item[(ii)] the subset $G.X_0$ is Borel and conull in $X$.
\end{itemize}
A subset $X_0\subset X$ with the above mentioned properties is called a \emph{cross section} for the action. 
Note that since $G$ is assumed second countable, if $\theta\restriction \colon U\times X_0 \to X$ is injective then $\theta\restriction\colon G\times X_0\to X$ is countable-to-one, so that $G.X_0$ is automatically Borel. It is not difficult to prove that $\R_{X_0}:=\{(x,x')\in X_0\times X_0\mid x'\in G.x\}$ is a countable Borel equivalence relation and that $X_0$ admits a probability measure $\mu_0$ with respect to which $\R_{X_0}$ is non-singular and whose class is compatible with the class of $\mu$ in a suitable sense. The existence of $\mu_0$ can either be derived as in the proof of \cite[Proposition 4.3]{KPV}, or  from \cite{feldman-hahn-moore} as follows:  If $X_0\subset X$ is a cross section  and $X'$ denotes the free part of $X$ then $X_0\cap X'$ is a cross section for the  restricted action $G\curvearrowright (X',\mu)$ as well as for $G\curvearrowright (X,\mu)$, and hence we may assume that the original action is free.   
Denote by $\G$ the action groupoid $G\ltimes X$. 
Since the action $G\curvearrowright (X,\mu)$ is free, a cross section as defined above is exactly the same as a \emph{lacunary section} as defined in \cite{feldman-hahn-moore} and by \cite[Theorem 5.3]{feldman-hahn-moore} (and its proof) the restriction $\G\restriction_{X_0}\simeq \R_{X_0}$  is a discrete reduction; thus, in particular, a measure groupoid in its own right. Denote by $[\lambda_0]$ the class of the Haar measure on $\G\restriction_{X_0}$ and by $\mu_0$ a probability representative for the induced measure class  $[r_*\lambda_0]$ on the unit space $X_0$ of $\G_0$; then $\R_{X_0}$ preserves $[\mu_0]$ by \cite[Remark 3.2]{feldman-hahn-moore}.  
 In the situation where $G$ is a countable group, the entire space $X$ serves as a cross section and hence $\G\simeq \R_{X_0}$,  and in the case where $G$ is non-discrete, \cite[Theorem 5.6]{feldman-hahn-moore} shows that
$
\G\simeq \T\times \R_{X_0}.
$
Since $\T$ can be realized as the action groupoid of the multiplication action $\rS^1\curvearrowright (\rS^1,\lambda_{\rS^1})$ and $\R_{X_0}$ can be realized as the orbit equivalence relation of an action of a discrete group  $\Gamma$ (see \cite[Therorem 1]{FM1}), Lemma \ref{lem:OE-and-iso} and Remark \ref{rem:non-free-iso-implies-OE} now show that the the action $G\curvearrowright (X,\mu)$ is orbit equivalent to $\rS^1\times \Gamma \curvearrowright (\rS^1\times X_0,\lambda_{\rS^1}\times \mu_0)$; see also \cite[Theorem 1.15]{carderi-le-maitre} for the corresponding result in the unimodular case.\\

\subsection{von Neumann algebraic aspects}\label{sec:vNa}
As shown in \cite{hahn-reg-rep}, every measure groupoid $(\G, C)$ naturally gives rise to a  von  Neumann algebra $L(\G)$ and in the case where $\G=G\ltimes X$ for a pmp action 
$G\curvearrowright (X,\mu)$ the von Neumann algebra $L(\G)$ is isomorphic to the crossed product von Neumann algebra $\Linfty(X)\rtimes G$ and is therefore a factor if the action is essentially free and ergodic \cite[XIII, Corollary 1.6]{takesaki-3}. The following proposition, which is probably well known to experts, relates various notions of amenability.  For definitions and basic properties we refer to: \cite[Chapter XVI]{takesaki-3} for the von Neumann algebraic setting,   \cite[Appendix G]{BHV} for group theoretical setting  and \cite{Connes-Feldman-Weiss} for the setting of countable equivalence relations.

\begin{prop}\label{prop:amenable-and-injective}
If $G$ is a non-discrete lcsc group, $G\curvearrowright (X,\mu)$ is an essentially free, ergodic  pmp action, and $X_0\subset X$ is a cross section for the action then
$
\Linfty(X)\rtimes G \simeq \BB(\ell^2(\NN))\htens L(\R_{X_0})
$
 and the following are equivalent:
 \begin{enumerate}
 \item $G$ is amenable.
  \item $\Linfty(X)\rtimes G$ is injective.
  \item $L(\R_{X_0})$ is injective.
  \item $\R_{X_0}$ is hyperfinite.
  \end{enumerate}
 \end{prop}

\begin{proof}
As shown in Section \ref{subsec:cross-sections}, the action groupoid $\G:=G\ltimes X$ splits as $\T \times \R_{X_0}$ and by \cite[Proposition 8.6]{feldman-hahn-moore} the associated von Neumann algebra therefore splits as a tensor product; i.e.
\[
\Linfty(X)\rtimes G \simeq L(\G)\simeq L( \T) \htens L(\R_{X_0})\simeq \BB(\ell^2(\NN)) \htens L(\R_{X_0}).
\]
Since  $\BB(\ell^2(\NN)) \htens L(\R_{X_0})$ is injective iff $ L(\R_{X_0})$ is injective, the equivalence of $(2)$ and $(3)$ follows.   Assuming that $G$ is amenable, \cite[Theorem 8.10]{feldman-hahn-moore} shows that $L(\G)$ is AF; equivalently injective by Connes' results, see~\cite[XVI]{takesaki-3}. Conversely, assuming $L(\G)$ injective, by \cite[Corollary 6.2.12]{anantharaman-renault} this implies that $\G$ is amenable \cite[Definition 3.2.8]{anantharaman-renault}, as 
the action is assumed free. Using \cite[Theorem 4.2.7]{anantharaman-renault}, it is not difficult to see that if $\G$ is amenable then action $G\curvearrowright (X,\mu)$ is amenable in the sense of \cite[Definition 1.4]{zimmer-amenable-ergodic-group-actions}, and since the action is assumed essentially free, pmp and ergodic, amenability of the action is equivalent with amenability of $G$ by  \cite[Theorem 2.1, Proposition 4.3 \& 4.4]{zimmer-amenable-ergodic-group-actions}. Thus, the equivalence of (1), (2) and (3) follows. 
 Lastly, we need to show the equivalence between (3) and (4),  which follows from Proposition 7.1 (and the remark following it) and Theorem 10 in \cite{Connes-Feldman-Weiss}.
\end{proof}
\begin{rem}
As the proof just given shows, under the assumptions in  Proposition \ref{prop:amenable-and-injective} the four equivalent conditions are furthermore equivalent with amenability of the action groupoid  \cite[Definition 3.2.8]{anantharaman-renault} as well as with amenability of the action  \cite[Definition 1.4]{zimmer-amenable-ergodic-group-actions}.  
\end{rem}

Injective factors are well understood via their so-called \emph{$\rS $-invariant} \cite{con76, connes-type-III-classification, uffe-bicentralizer}, which we will now describe in the special case of a free, ergodic pmp action $G\curvearrowright (X,\mu)$ of an lcsc group $G$ on a standard Borel probability space $(X,\mu)$. The faithful state $\varphi\colon \Linfty(X)\to \CC$  given by integration against $\mu$ gives rise to a dual weight $\tilde{\varphi}$ on the crossed product $\Linfty(X)\rtimes G$ \cite[X, Definition 1.16]{takesaki-2} and the modular operator $\Delta_{\tilde{\varphi}}$ has spectrum 
$\Sp(\Delta_{\tilde{\varphi}})=\overline{\Delta_G(G)}$, where the the latter denotes the closure in $\RR$ of the image of the modular function \cite[Chapter X]{takesaki-2}. If $(\sigma_t^{\tilde{\varphi}})_{t\in \RR}$ denotes the modular automorphism group associated with $\tilde{\varphi}$, the centralizer of the dual weight is defined as $\{x\in \Linfty(X)\rtimes G \mid \forall t\in \RR: \sigma_t^{\tilde{\varphi}}(x)=x \}$ and this subalgebra is isomorphic 
with the crossed product $\Linfty(X)\rtimes G_0$ where $G_0:=\ker(\Delta_G)$. If this is also a factor (equivalently, if the action of $G_0$ is also ergodic) then the $\rS $-invariant of the crossed product is  given by \cite[Corollaire 3.2.7]{connes-type-III-classification}
\[
\rS (\Linfty(X)\rtimes G)=\Sp(\Delta_{\tilde{\varphi}})=\overline{\Delta_G(G)}.
\]

\begin{rem}\label{rem:mp-of-action-vs-relation}
As a consequence of the discussions above, we note  that measure preservingness of an action $G\curvearrowright (X,\mu)$ of a non-discrete group on a standard probability space is not  a property of the induced orbit equivalence relation as it is the case for actions of discrete groups. 
To see this,  choose a free, (strongly) mixing \cite{klaus-schmidt} pmp action $G\curvearrowright (X,\mu)$, which is known to exist, e.g.~by  \cite[Remark 1.1]{KPV}, of a non-unimodular group $G$.  Picking a cross section $X_0\subset X$, Lemma \ref{lem:OE-and-iso}, Remark \ref{rem:non-free-iso-implies-OE} and the results in Section \ref{subsec:cross-sections} show that $G\curvearrowright (X,\mu)$ is orbit equivalent with an action $\rS^1\times \Gamma\curvearrowright (\rS^1\times X_0,\lambda_{\rS^1}\times \mu_0)$ such that $\R_{X_0}=\R_{\Gamma \curvearrowright X_0}$.  We argue that $\R_{X_0}$ is not pmp and hence $\rS^1\times \Gamma\curvearrowright (\rS^1\times X_0,\lambda_{\rS^1}\times \mu_0)$ is not pmp.

Indeed, as $G$ is non-compact so is $\ker(\Delta_G)$ \cite[(38.26)]{hewitt-ross} and it therefore still acts ergodically. Thus \cite[Corollaire 3.2.8]{connes-type-III-classification} implies that
\[
\rS \left(L(\R_{X_0})\right)=\rS \left(\BB(\ell^2(\NN))\bar{\otimes} L(\R_{X_0})\right)=\rS \left(\Linfty(X)\rtimes G\right)=\overline{\Delta_G(G)}\neq \{1\},
\]
and $L\R_{X_0}$ is therefore not type II$_1$, and $\R_{X_0}$  not pmp.
\end{rem}

\section{Measure equivalence}
In this section we begin our study of measure equivalence as introduced by  Bader-Furman-Sauer for unimodular lcsc groups in  \cite{BFS-integrable}  and by
Deprez-Li in \cite{deprez-li-ME} for general lcsc groups.

\begin{defi}[{\cite[Definition 3.1]{deprez-li-ME}}]
Let $G$ and $H$ be lcsc groups. A non-singular  standard Borel $G\times H$-measure space $(\Omega,\eta)$ is called a \emph{measure $G$-$H$ correspondence} if there exist standard Borel probability spaces $(X,\mu)$ and $(Y,\nu)$ and 
\begin{itemize}
\item A non-singular isomorphism of measure $G$-spaces $i\colon (G\times Y, \lambda_G\times \nu) \to (\Omega, \eta)$ where $G\times Y$ is considered a $G$-space for the left multiplication action on the first leg.
\item A non-singular isomorphism of measure $H$-spaces $j\colon (H\times X, \lambda_H\times \mu) \to (\Omega, \eta)$ where $H\times X$ is considered a  $H$-space for the left multiplication action on the first leg.
\end{itemize}
\end{defi}
Note that such correspondences always exist since, for given $G$ and $H$, one may simply take $(\Omega,\eta):=(G\times H, \lambda_G\times \lambda_H)$ endowed with the translation action, so in order to get an interesting class of couplings with which to define a notion of measure equivalence, one needs to demand further structure. To this end, note that
a measure correspondence gives rise to non-singular near actions $G\curvearrowright (X,\mu)$ and $H\curvearrowright (Y,\nu)$ and associated Borel cocycles $\omega_G\colon H\times Y\to G$ and $\omega_H\colon G\times X\to H$ which are defined (almost everywhere) by the requirements that
\begin{align*}
i(g\omega_G(h,y)^{-1}, h.y) := h.i(g,y) \ \text{ and }  j(h\omega_H(g,x)^{-1}, g.x):=g.j(h,x).
\end{align*}
Recall that from \cite[Definition 2.10]{KKR17} that a near action is a composition law $G \times X \to X$ for which all axioms of an action hold only almost everywhere.  For more details on the construction of these maps we refer to \cite[Section 2]{KKR17}. The following notion of measure equivalence was suggested by Deprez and Li in \cite{deprez-li-ME}.

\begin{defi}[{\cite[Definition 3.5]{deprez-li-ME}}]\label{def:ME}
A measure $G$-$H$ correspondence $(\Omega, \eta, X,\mu, Y,\nu, i,j)$ is said to be  a \emph{$G$-$H$ measure equivalence coupling} if there exist probability measures $\mu'\in [\mu]$ and $\nu'\in [\nu]$ with respect to which the induced near actions of $G$ and $H$ on $X$ and $Y$, respectively, are pmp.  If a $G$-$H$ measure equivalence coupling exists then $G$ and $H$ are said to be \emph{measure equivalent} (ME). A measure $G$-$H$ correspondence $(\Omega, \eta, X,\mu, Y,\nu, i,j)$ is called a  \emph{$G$-$H$ measure subgroup coupling} if there exists a $G$-invariant probability measure in $[\mu]$, and in this case $G$ is said to be a \emph{measure equivalence subgroup} of $H$.
 \end{defi}

\begin{rem}\label{rem:ME-defi-comparison}
Similarly to the proof in the unimodular case \cite[Appendix A]{BFS-integrable}, one verifies  that measure equivalence is indeed an equivalence relation, and
 for  unimodular lcsc groups the two definitions coincide, as shown in \cite[Proposition 3.6]{deprez-li-ME}. 
A natural source of examples comes from the following situation:  if $G$ is an lcsc group and $H\leq G$ is a closed subgroup with $\Delta_G\restriction_H=\Delta_H$ then by \cite[Corollary B.1.7]{BHV} there exists a $G$-invariant measure on $G/H$ which is unique up to scaling,  and if this measure is finite (which happens for instance when $H$ is cocompact in $G$)  then $G$ and $H$ are ME. In this situation,  $H$ is said to be of \emph{finite covolume} in $G$.
\end{rem}

\begin{rem}\label{rem:tychomorphism}
In \cite{gheysens-monod}, Gheysens and Monod  introduced the notion of a \emph{tychomorphism} between lcsc groups and for the sake of completeness we now show that measure equivalence subgroups are essentially the same as tychomorphisms. If $G$ and $H$ are lcsc groups, one easily checks that a tychomorphism \cite[Definition 14]{gheysens-monod} from $H$ to $G$ is
also a $H$-$G$ measure subgroup coupling. 
Conversely, if $(\Omega,\eta, X,\mu, Y,\nu, i,j)$ is a $H$-$G$ measure subgroup coupling, we may assume that $\nu$ is an $H$-invariant probability measure and hence obtain that $\eta':=i_*(\check{\lambda}_G\times \nu)$ is an $H$-invariant measure in the class of $\eta$ (here $\check{\lambda}_G$ denotes the push forward of $\lambda_G$ via the inversion map). Applying \cite[Lemma 2.10]{KKR17} we now obtain a Borel function $b\colon X\to [0,\infty[$ such that $(i^{-1})_*\eta'=\lambda_H \times b\mu$ and hence $(\Omega, \eta', X, b\mu, Y, \nu, i,j)$ is a tychomorphism from $H$ to $G$. It therefore follows from  \cite[Theorem B]{gheysens-monod} that an lcsc groups is non-amenable if and only if it has a non-abelian free measure equivalence subgroup.
\end{rem}

To work efficiently with measure equivalence we need to show that one can always obtain (essentially) free and ergodic couplings. To this end we need the following lemma, which is surely well known to experts in ergodic theory, but for which we were unable to find a suitable reference. Recall than any probability measure $\nu$ on a standard Borel $G$-space $X$ allows for an essentially unique decomposition into ergodic components $\nu = \int_Z \nu_z \d p_* \nu (z)$ for some $G$-invariant Borel map $p: X \to Z$ onto a standard probability space as described in \cite[Section V]{Dang-decomposition}.  

\begin{lem}\label{lem:eq-ergodic-fibers}
  Let $G$ be an lcsc group and let $\rho$ and $\eta$ be quasi-invariant, equivalent probability measures on a standard Borel $G$-space $X$ and let $\int_Z \rho_z \d p_* \rho(z)$ be an ergodic decomposition of $\rho$. 
  Denoting a Borel representative of the Radon-Nikodym derivative 
$\tfrac{\d\eta}{\d\rho}\colon X\to ]0,\infty[$ by $f$ and defining 
$g\colon Z\to ]0,\infty[$ as $g(z):=\int_X f(x)\d\rho_z(x)$ then $g$ is measurable and $\d\eta_z:=\tfrac{1}{g(z)}f\d \rho_z$ is an ergodic decomposition of $\eta$. In particular, $\eta_z$ and $\rho_z$ are equivalent for almost all $z\in Z$.
\end{lem}

\begin{proof}
The fact that $g\colon Z\to [0,\infty] $  is measurable is contained in the definition of a decomposition, and since we have
\[
\int_{Z}g(z)\d p_*\rho(z)= \int_{Z}\int_X f(x)\d\rho_z(x)\d p_*\rho(z)=\int_X f(x)\d \rho(x)=\eta(X)=1,
\]
we also have that $g$ is finite $p_*\rho$-almost everywhere. Similarly, since $f>0$ and each $\rho_z$ is a probability measure it follows that $g$ is positive $p_*\rho$-almost everywhere.
The measure $\eta_z$ is then well-defined $p_*\rho$-almost everywhere and, by construction, each $\eta_z$ is a probability measure supported in $p^{-1}(z)$ since this is the case for $\rho_z$ (cf.~\cite[Theorem V.1]{Dang-decomposition}). Moreover, since $f>0$ and $\rho_z$ is ergodic so is $\eta_z$ so we only need to prove that $\eta=\int_Z \eta_z d(p_*\eta)(z)$.  Since $\eta \sim \rho$ we have $p_*\eta \sim p_*\rho$ and we first show that $\tfrac{\d (p_*\eta)}{\d (p_*\rho)}=g$. To see this, we fix a Borel function $h\colon Z \to [0,\infty[$ and compute:
\begin{align*}
\int_Z h(z)g(z)\d (p_*\rho)(z) &= \int_Z \int_X h(z)f(x) \d \rho_z(x) \d (p_*\rho)(z)\\
&= \int_X \int_X h(p(x'))f(x) \d \rho_{p(x')}(x) \d \rho(x') \\
&= \int_X \int_X h(p(x'))f(x) 1_{p^{-1}(p(x'))}(x) \d \rho_{p(x')}(x)   \d \rho(x') \\
&= \int_X \int_X h(p(x))f(x) 1_{p^{-1}(p(x'))}(x) \d \rho_{p(x')}(x)   \d \rho(x') \\
&= \int_X \int_X h(p(x))f(x)  \d \rho_{p(x')}(x)   \d \rho(x') \\
&= \int_Z \int_X h(p(x))f(x)  \d \rho_{z}(x)   \d (p_*\rho)(z) \\
&=  \int_X h(p(x))f(x)  \d \rho(x) \\
&=\int_Z h(z) \d (p_*\eta)(z).
\end{align*}
Thus $g= \tfrac{\d (p_*\eta)}{\d (p_*\rho)}$ as claimed and from this it will now follow that $\eta= \int_{Z} \eta_z \d (p_*\eta)(z)$. To see the latter, put $\sigma:=\int_{Z} \eta_z \d (p_*\eta)(z)$ and let $h\colon X\to [0,\infty[$ be Borel; then
\begin{align*}
\int_{X} h(x)\d \sigma(x) &= \int_Z \int_X h(x) \d\eta_z(x)\d (p_*\eta)(z)\\
&= \int_Z \int_X h(x) \frac{f(x)}{g(z)} \d\rho_z(x)\d (p_*\eta)(z)\\
&= \int_Z \int_X h(x) \frac{f(x)}{g(z)} \d\rho_z(x) g(z)\d (p_*\rho)(z)\\
&= \int_X h(x) f(x) \d\rho(x)\\
&= \int_X h(x) \d\eta(x) \qedhere
\end{align*}
\end{proof}

Recall from \cite{KKR17} that an ME coupling $(\Omega,\eta,X,\mu,Y,\nu,i,j)$ is called \emph{strict} if $i$ and $j$ are equivariant Borel isomorphisms.
\begin{prop}\label{prop:strict-free-couplings}
If $G$ and $H$ are  measure equivalent then they admit an ME-coupling which is strict and on which $G\times H$ acts freely and ergodically.
\end{prop}
\begin{proof}
The proofs of \cite[Theorem 2.8 and Proposition 2.13]{KKR17} are easily adapted to the non-unimodular case, to show that one can obtain a coupling $(\Omega,\eta, X, \mu, Y,\nu, i,j)$ which is both strict and free.
We assume, as we may, that the  measures $\mu$ and $\nu$ are probability measures invariant under the actions of $G$ and $H$, respectively, and we furthermore fix probability measures $\tilde{\lambda}_G \in [\lambda_G]$ and $\tilde{\lambda}_H\in [\lambda_H] $. Decomposing $\mu$ and $\nu$ into ergodic, invariant probability measures as $\mu=\int_{Z'}\mu_{z'}\d\zeta'(z')$ and $\nu=\int_{Z''} \nu_{z''} d\zeta''(z'')$ \cite[Theorem 4.2]{varadarajan}
we also obtain ergodic decompositions of $i_*(\tilde{\lambda}_G\times \nu)=\int_{Z''} i_*(\tilde{\lambda}_G\times \nu_{z''}) \d\zeta''(z'')$ and $j_*(\tilde{\lambda}_H\times \mu)=\int_{Z'} j_*(\tilde{\lambda}_H\times \mu_{z'}) \d\zeta'(z')$. Since $i_*(\tilde{\lambda}_G\times \nu)\sim j_*(\tilde{\lambda}_H\times \mu)$, Lemma \ref{lem:eq-ergodic-fibers} at the same time provides us with decompositions of  $i_*(\tilde{\lambda}_G\times \nu)$ and  $ j_*(\tilde{\lambda}_H\times \mu)$ over a common standard Borel probability space $(Z,\zeta)$, and by the uniqueness of the ergodic decomposition \cite[Theorem V.1]{Dang-decomposition} we conclude that for almost all $z'\in Z'$ and almost all $z''\in Z''$ we have $i_*(\tilde{\lambda}_G\times \nu_{z''}) \sim j_*(\tilde{\lambda}_H\times \mu_{z'})$; thus, endowing $X$ and $Y$ with such $\mu_{z'}$ and $\nu_{z''}$ and $\Omega$ with $j_*(\tilde{\lambda}_H\times \mu_{z'})$ provides the strict, free, ergodic coupling.\end{proof}

Note that when $\Omega$ is a strict measure equivalence coupling then the near actions and their associated measurable cocycles are genuine actions and cocycles defined by the relations
\[
\omega_{G}(h,y)^{-1}:=\pi_G(i^{-1}(h.i(e_G, y))) \ \text{ and } \ h.y:=\pi_Y(i^{-1}(h.i(e_G, y))),
\]
and similarly for the $G$-action and its cocycle.  Moreover, when the action $G\times H\curvearrowright (\Omega,\eta)$ is either essentially free or ergodic then the same holds true for the induced actions $G\curvearrowright (X,\mu)$ and $H\curvearrowright(Y,\nu)$, and vice versa.\\

The following theorem extends \cite[Theorem 3.8]{KKR17}.  We restrict attention to non-discrete groups, but remark that since no unimodular lcsc group can be measure equivalent to a non-unimodular one (see Corollary \ref{cor:unimod-vs-non-unimod} below), Theorem  \ref{thm:me-and-soe-equivalence} together with \cite[Theorem 3.8]{KKR17} completely covers the class of lcsc groups.

\begin{thm}\label{thm:me-and-soe-equivalence}
For non-discrete, lcsc groups $G$ and $H$ the following are equivalent.
\begin{itemize}
\item[(i)] $G$ and $H$ are measure equivalent.
\item[(ii)] $G $ and $H $ admit  orbit equivalent, essentially free, ergodic,  pmp actions on standard Borel probability spaces.
\item[(iii)] $G$ and $H$ admit essentially free, ergodic pmp actions on standard Borel probability spaces for which the cross section equivalence relations associated with some choice of cross sections are stably orbit equivalent.
\end{itemize}
Moreover, in (iii) the word `some' may be replaced with `any'.
\end{thm}

\begin{proof}
Assuming $G$ and $H$ measure equivalent, we may, by Proposition \ref{prop:strict-free-couplings},  choose a strict, free, ergodic measure coupling $(\Omega, \eta, X,\mu, Y,\nu, i,j)$ so that the induced measure preserving actions  $G\curvearrowright (X,\mu)$ and $H\curvearrowright (Y,\nu)$ are then also free and ergodic.  
Put $\G=G\ltimes X$, $\H=H\ltimes Y$ and $\L=(G\times H)\ltimes \Omega$. The map 
\[
\Omega \times\Omega \xrightarrow{j^{-1}\times j^{-1}}  H\times X \times H \times X  \xrightarrow{\id_H \times \sigma \times \id_X}  H\times H \times X \times X
\]
(where $\sigma$ denotes  the coordinate flip) conjugates the orbit equivalence relation $\R_{G\times H\curvearrowright \Omega}$ onto a direct product of the transitive relation on $H$ and $\R_{G\curvearrowright X}$, and we therefore get an isomorphism $\L\simeq \T \times \G$ and, by symmetry, also an isomorphism $\L\simeq \T \times \H$. Thus, $\G$ and $\H$ are  similar and since both have uncountable orbits, \cite[Corollary 5.8]{feldman-hahn-moore} implies that $\G$ and $\H$ are isomorphic and the two actions $G\curvearrowright (X,\mu)$ and $H\curvearrowright (Y,\nu)$ are therefore orbit equivalent by Lemma \ref{lem:OE-and-iso}; thus (i) implies (ii). Assuming the existence of orbit equivalent, essentially free, ergodic actions $G\curvearrowright (X,\mu)$ and $H\curvearrowright (Y,\nu)$, Lemma \ref{lem:OE-and-iso} shows that the associated action groupoids $\G$ and $\H$ are isomorphic,
and upon 
choosing cross sections $X_0\subset X$ and $Y_0\subset Y$ for the two actions,  \cite[Theorem 5.6]{feldman-hahn-moore} (and its proof) shows that $\G\simeq \T\times \R_{X_0}$ and $\H\simeq \T\times \R_{Y_0}$, and hence $\R_{X_0}$ is similar to $\R_{Y_0}$.  Since both are countable equivalence relations, \cite[Proposition 4.14 and Theorem 4.12]{feldman-hahn-moore} show that $\R_{X_0}\times \Z \simeq \R_{Y_0}\times \Z$ and since the actions $G\curvearrowright (X,\mu)$, $H\curvearrowright (Y,\nu)$ are ergodic so are $\R_{X_0}$, $\R_{Y_0}$ and it therefore follows from \cite[Theorem 3]{FM1} that $\R_{X_0}$ and $\R_{Y_0}$ are SOE. Thus, (ii) implies that (iii) holds for any choice of cross sections.
Assume now that $G$ and $H$ admit essentially free, ergodic pmp actions $G\curvearrowright (X,\mu)$ and $H\curvearrowright (Y,\nu)$ for which there exist cross sections $X_0\subset X$ and $Y_0\subset Y$ for which the associated cross section equivalence relations are SOE. This means that $\R_{X_0} \times \Z \simeq \R_{Y_0}\times \Z$ and since $\Z\times \T\simeq \T$,    we now get (using again \cite[Theorem 5.6]{feldman-hahn-moore}) that:
\[
G\ltimes X \simeq \R_{X_0}\times \T \simeq \R_{X_0} \times \Z\times \T\simeq \R_{Y_0}\times \Z\times \T\simeq \R_{Y_0}\times \T\simeq H \ltimes Y,
\]
so by Lemma \ref{lem:OE-and-iso} the two actions  $G\curvearrowright (X,\mu)$ and $H\curvearrowright (Y,\nu)$ are OE; thus (iii) implies (ii). \\

Finally, we need to prove that (ii) implies (i).  
Let $G\curvearrowright (X,\mu)$ and $H\curvearrowright (Y,\nu)$ be orbit equivalent essentially free, ergodic pmp actions. By \cite[Lemma 10]{stefaan-sven-niels} we may assume that both actions are  free, and by Lemma \ref{lem:OE-and-iso} (and its proof) we get a strict isomorphism $\Phi\colon G\ltimes X\restriction_{X_0} \to H\ltimes Y\restriction_{Y_0} $ between inessential reductions, where $X_0, Y_0$ are the conull subsets from the definition of OE; more precisely,  $\Phi(g,x)=(c(g,x), \Delta(x))$ where $\Delta$ is the orbit equivalence and $c$ satisfies the cocycle relation whenever this makes sense.
The map $\Phi$ therefore dualizes to a isomorphism of von Neumann algebras $\Phi^*\colon \Linfty(H\times Y) \to \Linfty(G\times X)$, whose inverse at the spatial level, $\Psi \colon H\ltimes Y\restriction_{Y_0}\to G\ltimes X\restriction_{X_0}$,   is given by  by the formula $\Psi(h,y)=(d(h,y), \Delta^{-1}(y))$ for a cocycle $d$; cf.~\cite[Lemma 3.11]{KKR17} for more details.
We now define two measurable left actions of $G$ on $G \times X$ and of $H$ on $H \times Y$ by setting
\begin{align*}
  g' \Cdot (g,x) & := (g'g, x) \ \text{ and } \ g'\triangleright (g,x) := (g g'^{-1}, g'.x)
           && \quad \text{for all } g',g \in G \text{ and } x \in X, \\
h'\Cdot (h, y)  & := ( h' h ,y)   \ \text{ and } \         h'\triangleright  (h, y) := ( h h'^{-1} , h'. y)           &&  \quad \text{for all } h',h \in H \text{ and } y \in Y\text{.}
\end{align*}
Denoting by $I_G$ the groupoid inversion in $G\ltimes X$ given by $I_G(g,x)=(g^{-1},g.x)$ we see that $I_G(g'\Cdot (g,x))=g'\triangleright I_G(g,x)$ and since the $(G,\Cdot)$-action clearly preserves $\lambda_G\times \mu$, the $(G,\triangleright$)-action is non-singular, and similarly for the actions of $H$ on $H\times Y$.  The four actions therefore induce actions at the level of $\Linfty$-algebras and through the isomorphism $\Phi^*$ we therefore obtain induced actions
\[
G\overset{\Cdot}{\curvearrowright} \Linfty(H\times Y), \ H\overset{\Cdot}{\curvearrowright} \Linfty(G\times X), \  G\overset{\triangleright}{\curvearrowright} \Linfty(H\times Y) \ \text{ and } H\overset{\triangleright}{\curvearrowright} \Linfty(G\times X).
\]
A direct computation reveals that  for $h\in H$ and $f\in \Linfty(G\times X)$ we have
\[
(h\triangleright f) (g,x)= f\left((g d(h^{-1}, \Delta(x))^{-1}, d(h^{-1}, \Delta(x)). x)\right), \quad (g,x)\in G\times X,
\]
and from this it follows that the $\Cdot$-action of $G$ and the $\triangleright$-action of $H$ on $\Linfty(G\times X)$ commute,  and similarly for the corresponding actions on $\Linfty(H\times Y)$.  
We therefore obtain actions
\[
G\times H\overset{\Cdot \times \triangleright}{\curvearrowright} \Linfty(G\times X) \ \text{ and } \ G\times H \overset{\triangleright \times \Cdot }{\curvearrowright} \Linfty(H\times Y) 
\]
 and $\Phi^*\circ I_H^*\colon \Linfty(H\times Y)\to \Linfty(G\times X)$  intertwines the $(G\times H, \triangleright \times \Cdot)$-action on $\Linfty(H\times Y)$ with the $(G\times H, \Cdot\times \triangleright)$-action on $\Linfty(G\times X)$.   
 By \cite[Lemma 3.2]{ramsey} these actions (pre-)dualize to genuine actions at the level of spaces (denoted again by $\Cdot$ and $\triangleright$) and
 by \cite[Theorem 2]{mackey-point-realization} there exist a $(G\times H, \Cdot \times \triangleright )$-invariant, conull, Borel subset $\Omega_0 \subset G \times X$ and a  $(G\times H,  \triangleright \times \Cdot)$-invariant conull Borel subset  $\Omega_0' \subset H \times Y$ and a $G\times H$-equivariant Borel isomorphism $\varphi\colon \Omega_0 \rightarrow \Omega_0'$ which dualizes to $\Phi^*\circ I_H^*$.  Since $\Omega_0$ is invariant under the $(G,\Cdot)$-action, it is of the form $G \times X_1$ for some  conull Borel subset $X_1 \subset X$ and since $\Omega_0'$ is invariant for the $(H,\Cdot)$-action it is of the form $H\times Y_1$ for 
 some conull Borel subset $Y_1\subset Y$.
Thus defining $(\Omega, \eta):=(G\times X_1, \lambda_G\times \mu)$ with the $(G\times H, \Cdot\times \triangleright)$-action provides us with a measure correspondence since the map $\varphi^{-1}\colon (H \times Y_1, \lambda_H\times \nu) \to (\Omega, \eta)$ is a non-singular Borel isomorphism which intertwines the $(H,\Cdot)$-action on $H\times Y_1$ with the $(H,\triangleright)$-action on $\Omega.$ To show that $\Omega$ is indeed a measure equivalence coupling, we therefore need to find invariant probability measures in $[\mu]$ and $[\nu]$ for the induced actions of $H$ and $G$, respectively. However, since $\varphi$ intertwines the two $G$-actions, for $g\in G$ and almost all $(h,y)\in H\times Y_1$ we get
 \begin{align*}
 \varphi(g\Cdot \varphi^{-1}(h,y))&= g\triangleright (h,y)\\
 &=\left(hc\left(g,\Delta^{-1}(y)\right)^{-1}, c\left(g,\Delta^{-1}(y)\right).y \right)\\
 &=\left(hc\left(g,\Delta^{-1}(y)\right)^{-1}, \Delta(g.\Delta^{-1}(y)) \right).
 \end{align*}
The induced action $G\curvearrowright Y_1$ is therefore given  by $g.y= \Delta(g.\Delta^{-1}(y))$, and since
the original action $G\curvearrowright (X,\mu)$ is pmp, $\nu':=\Delta_*(\mu)\sim \nu$ is invariant for induced action; similarly, we get that $\mu':=\Delta^{-1}_*\nu\sim \mu$ is invariant for the induced action   of $H$ on $X_1$; 
thus, $(\Omega, \eta)$ is a measure equivalence coupling.

\end{proof}

\begin{cor}\label{cor:unimod-vs-non-unimod}
No unimodular group is measure equivalent to a non-unimodular  group. 
\end{cor}
\begin{proof}
Let $G$ and $H$ be lcsc groups, the former unimodular and  the latter non-unimodular, and assume that $G$ and $H$ are ME. If $G$ is discrete, then $G$ is a cocompact, normal  subgroup in the non-discrete unimodular group $G\times \mathrm{S}^1$, and thus measure equivalent to it by Remark \ref{rem:ME-defi-comparison}, so we may assume that both groups are non-discrete. Theorem \ref{thm:me-and-soe-equivalence} then provides us with orbit equivalent, essentially free, ergodic pmp actions $G\curvearrowright (X,\mu)$ and $H\curvearrowright (Y,\nu)$ and the associated action groupoids $G\ltimes X$ and $H\ltimes Y$ are therefore isomorphic by Lemma \ref{lem:OE-and-iso}. However,  by Lemma \ref{lem:unimod-of-groupoid}, $G\ltimes X$ is unimodular and $H\ltimes Y$ is not, and hence they cannot be isomorphic.
\end{proof}

Orbit equivalence for pmp actions of topological groups was also treated in \cite{carderi-le-maitre}, with the difference that the map witnessing the OE is required to be measure preserving, and not just non-singular.  As the following corollary shows, the two definitions agree on the class of unimodular lcsc groups. For countable discrete groups, this is well known and contained, for instance, in \cite[Remark 2.1]{Furman-OE-rigidity}. 

\begin{cor}\label{cor:mp-OE}
If $G$ and $H$ are unimodular lcsc groups with orbit equivalent, essentially free, ergodic, pmp actions $G\curvearrowright (X,\mu)$ and $H\curvearrowright (Y,\nu)$ on standard Borel probability spaces then there exists an orbit equivalence which is measure preserving.
\end{cor}

\begin{proof}
  First note that if one group is countable discrete then so is the other in which case the claim follows since $\Delta_*(\mu)$ is another probability measure in $[\nu]$ which is invariant with respect to the countable equivalence relation $\S$ generated by $H \curvearrowright (Y, \nu)$ and hence the Radon-Nikodym derivative $\frac{\d\Delta_*(\mu)}{\d\nu}$ is $\S$-invariant too, and thus constant by ergodicity. We may therefore assume that both groups are non-discrete. Pick cross sections $X_0$ and $Y_0$ for the two actions and endow them with the canonical probability measures $\mu_0$ and $\nu_0$ with respect to which the cross section equivalence relations are pmp and ergodic  \cite[Proposition 4.3]{KPV}. By Theorem \ref{thm:me-and-soe-equivalence},    
$\R_{X_0}$ and $\R_{Y_0}$ are SOE, and by passing to suitable subsets of $X_0$ and $Y_0$ we obtain new cross sections (by ergodicity of the actions), so we may as well assume that $\R_{X_0}$ and $\R_{Y_0}$ are orbit equivalent.  By the first part of the proof, this orbit equivalence $\Delta_0\colon X_0\to Y_0$ is therefore pmp, and by choosing countable groups $\Gamma$ and $\Lambda$ implementing $\R_{X_0}$ and $\R_{Y_0}$, respectively, it follows from  \cite[Theorem 1.15]{carderi-le-maitre} and its proof that
$
G\curvearrowright (X,\mu)$ is orbit equivalent with
\[
 \rS^1\times \Gamma \curvearrowright (\rS^1\times X_0, \lambda_{\rS^1}\times \mu_0 ),
\]
through a measure preserving OE, and similarly for the $H$-action. Since $\id\times \Delta_0$ is a measure preserving OE between
\[
 \rS^1\times \Gamma \curvearrowright (\rS^1\times X_0, \lambda_{\rS^1}\times \mu_0 ) \ \text{ and } \   \rS^1\times \Lambda \curvearrowright (\rS^1\times Y_0, \lambda_{\rS^1}\times \nu_0 )
\]
we conclude that the original actions also admit an OE which is measure preserving.
\end{proof}

\section{Measure equivalence of amenable groups}
In \cite[Theorem 6.1]{deprez-li-ME} it is proven that weak amenability, the Haagerup property and the weak Haagerup property  are all preserved by measure equivalence, and in \cite{KKR17} the authors prove that the same is true for Kazhdan's property (T),  and  for amenability under the additional assumption of  unimodularity. 
The aim of this section is to complete the classification of lcsc amenable groups up to measure equivalence, which is the content of Theorem \ref{thm:ME-classification-of-amenable} below. As a first step in this direction, we start out by proving that the class of amenable groups is closed under measure equivalence.  This is a rather direct consequence of Theorem \ref{thm:me-and-soe-equivalence} and \cite[Proposition 5.10]{zimmer-82} in conjunction, but since the latter is stated without proof, we give a short and self-contained proof for the benefit of the reader.

\begin{prop}\label{prop:amenability-transfers}
Measure equivalence preserves amenability.  
\end{prop}
\begin{proof}
By passing to a direct product with $\mathrm{S}^1$ if necessary, we may, without loss of generality, assume that the two groups are non-discrete.
If a compact group is ME with another group, then this group has to be unimodular
 by Corollary \ref{cor:unimod-vs-non-unimod} and hence
the two groups would admit a coupling $(\Omega,\eta, X,\mu,Y,\nu, i,j)$  in which the maps $i$ and $j$ are measure preserving (Remark \ref{rem:ME-defi-comparison}), and this can only be if
the other group is compact as well. Since compact groups are amenable, we may therefore also assume that neither group is compact.
If $G$ and $H$ are ME, Theorem \ref{thm:me-and-soe-equivalence} provide is with orbit equivalent, essentially free, ergodic, pmp actions $G\curvearrowright (X,\mu)$ and $H\curvearrowright (Y,\nu)$, and by picking cross sections $X_0\subset X$ and $ Y_0\subset Y$ we obtain from Proposition \ref{prop:amenable-and-injective} and Lemma \ref{lem:OE-and-iso} that
\[
\BB(\ell^2(\NN))\bar{\otimes} L(\R_{X_0})\simeq \Linfty(X)\rtimes G=L(G\ltimes X)\simeq L(H\ltimes Y)\simeq \BB(\ell^2(\NN))\bar{\otimes} L (\R_{Y_0}).
\]
Moreover, by Proposition \ref{prop:amenable-and-injective}, amenability of $G$ is equivalent with injectivity of $L(\R_{X_0})$ or, equivalently, with injectivity of $\BB(\ell^2(\NN))\bar{\otimes} L(\R_{X_0})$, and hence the statement follows.
\end{proof}

\begin{lem}\label{lem:same-range-implies-ME}
If $G$ and $H$ are non-unimodular, amenable, lcsc groups with $\overline{\Delta_G(G)}=\overline{\Delta_H(H)}$ then $G$ and $H$ are measure equivalent.
\end{lem}
\begin{proof}

Pick free, (strongly) mixing (cf.~\cite{klaus-schmidt}), pmp actions $G\curvearrowright (X,\mu)$ and $H\curvearrowright (Y,\nu)$; cf. \cite[Remark 1.1]{KPV} for the existence of such. Then any non-compact subgroup of  either group  also acts ergodically and  the centralizer of the weight $\tilde{\varphi}$ on $\Linfty(X)\rtimes G$ dual to $\varphi:=\int - \d\mu$ is isomorphic to $\Linfty(X)\rtimes \ker(\Delta_G)$ (see Section \ref{sec:vNa}), which is therefore a factor. 
Hence
\[
\S(\Linfty(X)\rtimes G)=\Sp(\Delta_{\tilde{\varphi}})=\overline{\Delta(G)}.
\]
Choosing a cross section $X_0\subset X$ we have $\Linfty(X)\rtimes G\simeq \BB(\ell^2(\NN))\bar{\otimes} L(\R_{X_0})$, and  $L(\R_{X_0})$ is therefore a factor as well,  and by
 \cite[Corollaire 3.2.8]{connes-type-III-classification} we have
\[
\S(L(\R_{X_0}))=\S(\BB(\ell^2(\NN))\bar{\otimes} L(\R_{X_0}))=\S(\Linfty(X)\rtimes G)=\overline{\Delta_G(G)}.
\]
Choosing a cross section $Y_0\subset Y$ for the $H$-action, we similarly obtain that $\S(L(\R_{X_0}))=\overline{\Delta_H(H)}$ and hence, by amenability of $G$ and $H$, that $L(\R_{X_0})$ and $L(\R_{Y_0})$ are injective factors (see Proposition \ref{prop:amenable-and-injective}) with the same $\S$-invariant. Moreover, since $G$ and $H$ are non-unimodular the common $\S$-invariant is neither $\{0\}$ nor $\{0,1\}$, and the two factors are therefore both type III$_\lambda$ for some $\lambda \in ]0,1]$.
 By the classification of injective factors \cite{con76, connes-type-III-classification, uffe-bicentralizer} we therefore conclude that $L(\R_{X_0})\simeq L(\R_{Y_0})$. Moreover, \cite[Corollary 11]{Connes-Feldman-Weiss} shows that the isomorphism may be composed with an automorphism of $L(\R_{Y_0})$ such that the resulting isomorphism maps $\Linfty(X_0)$ onto $\Linfty(Y_0)$, and  $\R_{X_0}$ and $\R_{Y_0}$ now follow  orbit equivalent by \cite[Theorem 1]{FM2}.
By Theorem \ref{thm:me-and-soe-equivalence} we therefore obtain that $G$ and $H$ are ME, as desired.
\end{proof}
The following result extends \cite[Theorem 4.1]{KKR17}, and in turn the classical results for discrete groups \cite{Dye-mp-trans-I, ornstein-weiss, Connes-Feldman-Weiss}, completing the classification of amenable, lcsc groups up to measure equivalence.

\begin{thm}\label{thm:ME-classification-of-amenable}
The class of amenable lcsc groups is stable under measure equivalence and consists of the following three measure equivalence classes: compact groups, non-compact unimodular amenable groups, and non-unimodular amenable groups.
\end{thm}

\begin{proof}
The stability statement is exactly the contents of Proposition \ref{prop:amenability-transfers}.
Clearly all compact groups are mutually ME as one may simply use their product as an ME coupling, and if a compact group $K$  were ME to a non-compact group $G$, then $G$ would be unimodular by Corollary \ref{cor:unimod-vs-non-unimod}, and hence $G$ and $K$ would admit a coupling $(\Omega,\eta, X,\mu,Y,\nu, i,j)$  in which the maps $i$ and $j$ are measure preserving (Remark \ref{rem:ME-defi-comparison}), but this cannot be since $\lambda_K\times \mu$ is finite and $\lambda_G\times \nu$ is infinite. The compact groups thus form one ME class. By \cite[Theorem 4.1]{KKR17} all non-compact, unimodular, amenable lcsc groups are pairwise ME which, together with Corollary \ref{cor:unimod-vs-non-unimod} and Proposition \ref{prop:amenability-transfers},  shows that the non-compact, unimodular, amenable lcsc groups form another ME class. 
We therefore only need to prove that all non-unimodular, amenable  groups are pairwise ME.
If $G$ and $H$ are lcsc amenable groups and both $\Delta_G(G)$ and $\Delta_H(H)$ are dense in $\RR_+$ then $G$ and $H$ follow ME by Lemma
\ref{lem:same-range-implies-ME}. If $G$ is non-unimodular and $\Delta_G(G)$ is not dense in $\RR_+$ then there exists $s_0\in \RR_+$ such that $\Delta_G(G)=\{s_0^n \mid n\in \ZZ\}=:S$, and since $G$ is non-unimodular we have $s_0\neq 1$. Consider now the subgroup of the $ax+b$-group given by 
\[
H=\left\{
\begin{bmatrix}
s & b\\
0 & 1
\end{bmatrix} \in \MM_2(\RR) \Big\vert \ b\in \RR, s\in S\right\}.
\]
It is easy to see 
that $\Delta_{H}\left(\begin{bmatrix}
s & b\\
0 & 1
\end{bmatrix}\right)=s^{-1}$ so that $\Delta_H(H)=S$ 
and  $H\simeq \RR\rtimes S$  is amenable, being an extension of the amenable groups $\RR$ and $S$; hence $G$ and $H$ are measure equivalent by Lemma \ref{lem:same-range-implies-ME}. Since $s_0\neq 1$, $H$ is a cocompact subgroup of the full $ax+b$-group
\[
\tilde{H}:=\left\{
\begin{bmatrix}
s & b\\
0 & 1
\end{bmatrix} \in \MM_2(\RR) \Big\vert \  b\in \RR, s\in \RR_{+}\right\}
\] 
and since $\Delta_{\tilde{H}}\left(\begin{bmatrix}
s & b\\
0 & 1
\end{bmatrix}\right)=s^{-1}$ (\cite[Example A.3.5]{BHV}) we have $\Delta_{\tilde{H}}\restriction_H=\Delta_H$ and hence $H$ and $\tilde{H}$ are ME as well; see Remark \ref{rem:ME-defi-comparison}. The group $G$ is therefore ME to a group whose modular function has dense (actually full) image, and we conclude from that  that all non-unimodular amenable groups are pairwise ME.
\end{proof}

\nocite{*}

\def\cprime{$'$} \def\cprime{$'$}

\end{document}